\documentclass[10pt,a4paper,twoside]{amsart}

\usepackage{amsfonts, amssymb, amsmath, amsthm, bm}
\usepackage{mathrsfs} 
\usepackage{latexsym}
\usepackage{enumerate}
\usepackage{multicol}
\usepackage{verbatim}
\usepackage{dsfont}

\usepackage{url}
\usepackage{csquotes}
\usepackage[%
    backend=bibtex,
    style=numeric-comp,
    citetracker=true,
    pagetracker=true,
    hyperref=true,
    backref=true,
    giveninits=true,
    bibencoding=utf8
    ]{biblatex}

\usepackage[colorlinks=true]{hyperref}
\hypersetup{citecolor=blue, linkcolor=blue}
\usepackage{mathabx} 

\newtheorem{thm}{Theorem}[section]
\newtheorem{lem}[thm]{Lemma}

\usepackage{showtags} 

\newcommand{\ve}{\varepsilon}
\newcommand{\Ocal}{\mathcal{O}}
\def\gP{\mathfrak{P}}
\def\gB{\mathfrak{B}}
\def\1{1\!\!\!1}

\title{Bounding the exponential sum on squares of some sifted sequences%
}
\author{E. Malavika}
\address[E. Malavika]{Indian Institute of Science Education and Research, 
Berhampur, 
Laudigam, Ganjam, 760003, India.}
\email{malavikae@iiserbpr.ac.in}
\author{Olivier Ramar\'e}
\address[O. Ramar\'e]{CNRS/ Institut de Math\'ematiques de Marseille, 
Aix Marseille Universit\'e, U.M.R. 7373, Site Sud, Campus de Luminy, 
Case 907, 13288 Marseille Cedex 9, France.}
\email{olivier.ramare@univ-amu.fr}

\begin{document}




\begin{abstract}
  Let $\gB$ denote the collection of odd primitive Gaussian integers and $n\mapsto b(n)$ denote the characteristic function of elements of $\gB$. We prove that the exponential sum $ S(\alpha; N)=\sum_{n\le N}b(n)e(n^2\alpha)$ satisfies 
   \begin{equation*}
       \frac{S(\alpha;N)}{N/\sqrt{\log N}}
       \ll  N^\epsilon (q^{-1/4}+N^{-1/2}q^{1/4}+N^{-1/8}),
  \end{equation*}
  where, $(a,q)=1$ and $|\alpha - a/q | < 1/q^2$. Though we specialized on sums of two squares, these results extend to more general sequences.
 
\end{abstract}

\maketitle


\section{Introduction and results}

In 1937 in~\cite{Vinogradov*37}, I.~M.~Vinogradov proved that every large enough
odd integer can be written as a sum of three primes, a result hitherto proved only
under the generalized Riemann Hypothesis by G.\,H.~Hardy \& J.\,R.~Littlewod
in~\cite{Hardy-Littlewood*22a}. A main step of Vinogradov's proof, clearly exposed in the monograph~\cite{Vinogradov*54}, is to bound the trigonometric polynomial 
over the primes $\sum_{p\le N}e(p\alpha)$. In~\cite{Vaughan*77b}, see 
also~\cite{Vaughan*80}, R.\,C.~Vaughan simplified this technique, and it is this simplification that A.~Ghosh used in~\cite{Ghosh*81} to bound the 
trigonometric polynomial $\sum_{p\le N}e(\alpha p^2)$. 

In \cite{Ramare-Viswanadham*24}, the second named author together with 
G.\,K.~Viswanadham considered the question of giving bounds for the
trigonometric polynomial over more general sifted sequences. As the techniques 
of Vinogradov or of Vaughan are not applicable here. The
primary example, which we also use in the present paper,
is the trigonometric polynomial of the
squares of the set~$\gB$ of odd integers that may be written
as sums of two coprime squares. We provide below bounds for this 
polynomial that are as good as the ones of A.~Ghosh.
As it turns out the techniques used
will prove to be
general enough to handle more cases, including the one of the primes, so our 
result generalizes the one of Ghosh while providing another access to it,
but let us consider in this first stage only one situation. The case of squares
appears also to deserve a special treatment among polynomials. Technically, 
this will come from the fact that the number of representations 
of an integer $j\le M^2$ as a difference of two squares~$m_1^2-m_2^2$ 
with~$m_1$ and~$m_2$ of about a same size~$M$ is small, which further implies
that the number of such $j$'s is large.

Let us notice that the elements of $\gB$ are all congruent to~1
modulo~4. This implies in particular that if a prime $p\equiv3[4]$
divides such an element, then at least another prime $\equiv 3[4]$ does so (which
may be the same, says $p$, meaning that $p^2$ divides our candidate).
As a conclusion, the odd integers below $N$ that are in~$\gB$ are the
integers below $N$ that are congruent to 1 modulo~4 and do not have
prime divisors $p\le \sqrt{N}$ congruent to~3 modulo~4.

Let us denote the characteristic function of elements of $\gB$ (we shall refer to them as \emph{odd primitive Gaussian integers})
by~$n\mapsto b(n)$.
Our aim is to understand
the trigonometric polynomial
\begin{equation}
  \label{defSalphaN}
  S(\alpha; N)=\sum_{n\le N}b(n)e(n^2\alpha).
\end{equation}

Here is our first result.
\begin{thm}
  \label{OneSum}
  Suppose that $|\alpha - a/q | < 1/q^2$, where $\alpha$ is a real number 
  with~$a$ and~$q$ integers satisfying $(a,q) = 1$. We have
  \begin{equation*}
    \frac{S(\alpha;N)}{N/\sqrt{\log N}}
    \ll  N^\epsilon (q^{-1/4}+N^{-1/2}q^{1/4}+N^{-1/8}).
  \end{equation*}
\end{thm}
The same bound holds true for the primes, recovering the result of A.~Ghosh.
We also prove the next estimate that proves useful when 
locating $\{\alpha b^2\}$ in intervals modulo~1.
\begin{thm}
  \label{SeveralSum}
  Notation being as in Theorem~\ref{OneSum}, we have
  \begin{equation*}
    \frac{\sqrt{\log N}}{N}
    \sum_{h\le H}|S(h\alpha;N)|
    \ll  H(Nq)^{\epsilon}\bigl( 
    q^{-1/4}+N^{-1/8}+H^{-1/4}N^{-1/2}q^{1/4}\bigr).
  \end{equation*}
\end{thm}

\subsection*{The trigonometric polynomial on Loeschian integers and more}
Let $\ell(n)$ be the characteristic function of those integers that may be 
written in the form $u^2+uv+v^2$, with $u$ and $v$ being coprime. These
integers are also those that have no prime factors $\equiv2[3]$, and are called the
primitive Loeschian integers. Our method directly applies to this case and both
Theorem~\ref{OneSum} and~\ref{SeveralSum} hold true. The same is also mutatis
mutandis true for
the trigonometric polynomial on integers that are both Gaussian and Loeschian, 
see the paper \cite{Fouvry-Levesque-Waldschmidt*18} by \'E.~Fouvry, C.~Levesque
and M.~Waldschmidt that explains the appearance of such a sequence. Another
interest is that this
latter sequence is of sieve dimension $3/4$, while the ones of Gaussian or 
Loeschian integers is of dimension~$1/2$ and the one of primes of dimension~$1$.

\subsubsection*{Notation}

It is typographically expedient to define
\begin{equation}
  \label{defgPz}
  \gP(z)=\prod_{\substack{p<z\\ p\in\gP}}p
\end{equation}
for any set of primes $\gP$ and the special case
$P_{4;3}(z)=\prod_{\substack{p<z\\ p\equiv 3[4]}}p$.

\section{Auxiliaries}
Here is a tool we borrow from \cite[Theorem
1.4]{Ramare-Viswanadham*24}.

\begin{lem}
  \label{OR-KV}
  For any set of parameters $\max(2,z)\le Z\le N$ and $2\le M_0\le M$,
  there exist sequences $(\alpha_\ell(t))$ and
  $(\beta_k(t))$ such that, for any sequence $(u_n)_{n\le N}$, we have
  \begin{multline*}
    \sum_{(n,\gP(Z))=1}\mkern-12mu u_n
    =\sum_{\substack{d<M\\d|\gP(z)}}\mu(d)\sum_{n\equiv 0[d]}u_n
    -\sum_{\substack{mp\le N\\ z\le p<
        Z}}
    \rho(m)\1_{p\in\gP}u_{mp}
    \\-
    \int_{0}^1
    \mkern-5mu\sum_{\substack{k\ell\ge M\\ M_0\le \ell\le M_0z}}\mkern-12mu
    \alpha_\ell(t)\beta_k(t)
    u_{k\ell}
    dt
    +
    \Ocal\biggl(
   \sqrt{ \frac{N\sum_{n}|u_n|^2}{z}}
    \biggr)
  \end{multline*}
  where $|\alpha_\ell(t)|\le 1$ and
  $|\beta_k(t)|\le 16\tau_3(k)\log N$ and $\alpha_\ell(t)$ vanishes on
  indices~$\ell$ that do not divide $\gP(z)$, while
  $\rho(m)\in[0,1]$. In case~$|u_n|\le 1$, the error term may be
  improved to $\Ocal(N/z)$.
\end{lem}

\begin{lem}
  \label{Vino}
 Suppose that $X\ge1$ and $Y\ge1$ are positive real numbers.
 Also suppose that $|\alpha - a/q | < 1/q^2$, where $\alpha$ is a real
 number with $a$ and $q$ integers satisfying $(a,q) = 1$. Then we have
 \begin{equation*}
   \sum_{n\le X}\min\bigl(Y,1/\|\alpha n\|\bigr)
   \ll \frac{XY}{q}+(X+q)\log 2q.
 \end{equation*}
\end{lem}
This is a classical lemma of Vinogradov in the form given in Lemma~1
of~\cite{Vaughan*77} by R.~C.~Vaughan. 

We end this section with a lemma that will help us treat the
condition~$X/2<mn\le X$. This (harmless) condition may often be
treated directly by inequalities. The next process is more general,
and is for instance used in Lemma~2 of~\cite{Vaughan*80}. 
\begin{lem}
    \label{FourierVaughan}
    For $x>0$ and $T>0$, we have
    \begin{equation*}
        \1_{x/2<\beta \le x}(\beta)
        =\int_{-T}^T e^{i\beta t}\frac{\sin(xt)-\sin(xt/2)}{\pi t}dt
        +
        \Ocal\biggl(\frac{1}{T\min(|\beta-x|,|\beta-x/2|)}\biggr).
    \end{equation*}
    We also have $\int_{-T}^T |\frac{\sin(xt)-\sin(xt/2)}{\pi t}|dt \ll x+\log \max(1,T)$.
\end{lem}

\section{On diagonal binary or quadric quadratic forms}
Let us start this section with a wide-ranging result. 
\begin{lem}
  \label{Vaughan-Wooley}
  Let $\epsilon>0$.
  The number of integral solutions of the equation
$ax^2+by^2=c$ with $abc\neq0$ and $1\le x,y\le P$ is $\ll_\epsilon (|abc|P)^\epsilon$.
\end{lem}
This is Lemma 3.5 of \cite{Vaughan-Wooley*95-1}.
Here is our next lemma, which is a degraded form of a specialisation to the diagonal case
of Theorem~1.1 of~\cite{Browning-Heath-Brown*18} by T.~Browning and
D.~R.~Heath-Brown. See also the papers \cite{Heath-Brown*96}, \cite{Browning*03}
and \cite{Browning*07}.
\begin{lem}
  \label{Browning-HB}
  Let $\epsilon>0$. Let us define
  \begin{equation*}
    Q((m_i)_{i\le 4},(h_i)_{i\le 4})
    =h_1m_1^2+h_2m_2^2-h_3m_3^2+h_4m_4^2
  \end{equation*}
  where $h_1,h_2,h_3,h_4>0$. Set $\Delta_Q=h_1h_2h_3h_4$ and
  $\|Q\|=\max(h_1,h_2,h_3,h_4)$. Set further
  \begin{equation*}
    \Delta_{\text{bad}}=\prod_{\substack{p^e\|\Delta_Q\\ e\ge 2}}p^e.
  \end{equation*}
  For $P\ge1$, the number $N(h_1,h_2,h_3,h_4,P)$ of solutions to
  we have $h_1m_1^2+h_2m_2^2=h_3m_3^2+h_4m_4^2$ that are such that
  $(m_1,m_2,m_3,m_4)\in[-P,P]^4$, $(m_1,m_2,m_3,m_4)=1$ satisfies
  \begin{equation*}
    N(h_1,h_2,h_3,h_4,P)
    \ll_\epsilon \Delta_{\text{bad}}^{1/4}
    \biggl(\frac{\|Q\|^4}{\Delta_Q}\biggr)^{5/8}
    \biggl(\frac{P^2}{\Delta_Q^{1/4}}+P^{4/3}\biggr)
    \log(2P)\log(2\Delta_Q)\Delta_{\text{bad}}^\epsilon
  \end{equation*}
  provided that $\Delta_{\text{bad}}^{20}\le P$.
\end{lem}
In the application we have in mind, the variables $h_i$ are all of a
same size~$H$. This is the topic of~\cite{Browning-Heath-Brown*20} by
T.~Browning and D.~R.~Heath-Brown, though with the size condition
$\max(h_i)^3\max(m_i)^2\le B$, which is not adapted to our
situation. However, the upper bound provided for $M_3$ in Lemma~2.1
of~\cite{Browning-Heath-Brown*20} contains the material we need.
\begin{lem}
  \label{Browning-HB2}
  Let $\epsilon>0$. 
  For $P\ge1$, the number $M_3(H,P)$ of solutions to
  we have $h_1m_1^2+h_2m_2^2=h_3m_3^2+h_4m_4^2$ that are such that
  $(m_1,m_2,m_3,m_4)\in[-P,P]^4$, $(h_1,h_2,h_3,h_4)\in[-H,H]^4$,
  $(m_1,m_2,m_3,m_4)=1$ satisfies 
  \begin{equation*}
    M_3(H,P)
    \ll_\epsilon
    H^3P^2+H^5P^{2/3}+(HP)^{2+\epsilon}.
  \end{equation*}
\end{lem}
The only supplementary step now is to remove the coprimality
condition, but this is readily done at a cost of a multiplicative constant.

\section{Lemmas for linear sums}

Let us start by part of Lemma~2 of \cite{Ghosh*81} by A.~Ghosh,
Eq. (3.10) in case~$H=1$. Let the parameters $V,W\leq N$ and $V'= \min \{ 2V, N\}$ and $W'=\min \{2W,N\}$. 
\begin{lem}
  \label{LinearGhosh}
  Assume that $|q\alpha-a|\le 1/q$ for some $a$ prime to~$q$.  Then,
  for any complex sequence $(a_m)$ bounded in absolute value by~1, we
  have
  \begin{equation*}
    \sum_{V< m\le V'}a_m
    \sum_{\substack{W<n\le W'\\ N/2 < mn\le N}}
    e(\alpha (mn)^2)
    \ll_\ve (Nq)^\ve
    \biggl(
    \frac{NV^{1/2}}{q^{1/2}}+N^{1/2}V+\sqrt{Vq}
    \biggr).
  \end{equation*}
\end{lem}

\begin{proof}
  Let us denote by $\Sigma$ our quantity.
  We use Cauchy's inequality to get
  \begin{align*}
    |\Sigma|^2
    &\ll
    V
    \sum_{V<m\le V'}\biggl|\sum_{\substack{W<n\le W'\\ N/2 < mn\le N}}
    e(\alpha (mn)^2)\biggr|^2
    \\&\ll
    V
    \biggl(
    W + \sum_{\substack{0\le |k|\le W'\\ V<m\le V'}}\biggl|
    \sum_{\substack{W<n\le W'\\ N/2 < mn\le N\\W<n+k\le W'\\ N/2 < m(n+k)\le N}}
    e(\alpha m^2(-2kn+k^2))\biggr|
    \biggr).
  \end{align*}
  This is obtained by expanding the squares, resulting in two
  variables $n_1$ and $n_2$, setting $k=n_1-n_2$ and summing over $k$
  and $n=n_1$. This is very much the process used when proving the
  Weyl-van der Corput inequality, see for instance~Lemma~2.5 of the
  book \cite{Graham-Kolesnik*91} by S.~W.~Graham \& G.~Kolesnik.
  This readily reduces to
  \begin{equation*}
    |\Sigma|^2
    \ll N
     + V\sum_{\substack{k\le W'\\ V<m\le V'}}\biggl|
    \sum_{\substack{n\in I(m,k)}}
    e(\alpha km^2n)\biggr|
  \end{equation*}
  for some interval $I(m,k)$ of length $\le W$. Therefore
  \begin{equation*}
    |\Sigma|^2
    \ll N
     + V\sum_{\substack{k\le W'\\ V<m\le V'}}\min(W, 1/\|km^2\alpha\|)
    .
  \end{equation*}
  A first possibility is to extend the variable $km^2$ to a variable
  $\ell\le W'V^2$ and to use Lemma~\ref{Vino}.
  Following this path, the reader will swiftly complete the proof.
\end{proof}

\begin{lem}
    \label{linear+congruence}
    Assume that $|q\alpha-a|\le 1/q$ for some $a$ prime to~$q$. Then,
  for any complex sequence $(a_m)$ bounded in absolute value by~1, we
  have
  \begin{equation*}
    \sum_{V< m\le V'}a_m
    \sum_{\substack{W<n\le W'\\ N/2 < mn\le N \\ mn \equiv 1 [4]}}
    e(\alpha (mn)^2)
    \ll_\ve (Nq)^\ve
    \biggl(
    \frac{NV^{1/2}}{q^{1/2}}+N^{1/2}V+\sqrt{Vq}
    \biggr).
  \end{equation*}
\end{lem}
\begin{proof}
 We use the orthogonality property of Dirichlet characters to handle the congruent condition, and to represent the principal and non principal Dirichlet characters modulo~4 we use $(1-e(x/2))/2$ and $(e(x/4)-e(3x/4))/2i$. 
    \begin{align}
    \label{toAdditive}
       \sum_{V< m\le V'}a_m\sum_{\substack{W<n\le W'\\ N/2 < mn\le N \\ mn \equiv 1 [4]}}
    &e(\alpha (mn)^2)
    = 
    \sum_{V< m\le V'}a_m \biggl(\sum_{\chi \mod{4}} \chi(m)\biggr)
    \notag\\
        &\times \sum_{\substack{W<n\le W'\\ N/2 < mn\le N }}
    e(\alpha (mn)^2) \biggl[ \frac{1-e(n/2)}{2}+
    \frac{e(n/4)-e(3n/4)}{2i}\biggr].
    \end{align}
    Here we have two types of sum
    \begin{equation*}
        \sum_{V< m\le V'}A_m
        \sum_{\substack{W<n\le W'\\ N/2 < mn\le N }}
        e(\alpha (mn)^2)
    \end{equation*}
    and
    \begin{equation*}
         \sum_{V< m\le V'}A_m
        \sum_{\substack{W<n\le W'\\ N/2 < mn\le N }}
        e(\alpha (mn)^2+\beta n)
    \end{equation*} 
    where $A_m$ is bounded in absolute by 1. We have a bound  for the first by
    Lemma~\ref{LinearGhosh}.  Let us denote the second by $\Sigma$.
    By applying Cauchy on $\Sigma$, we have
    \begin{align*}
        \Sigma 
        \ll& 
        \bigg( \sum_{V<m\leq V'}1^2\bigg)^{1/2}\bigg[ \sum_{V<m\leq V'}\bigg| \sum_{N/4V<n\leq N/V}e(\alpha(mn)^{2}+\beta n)\bigg|^{2}\bigg]^{1/2}\\
        \ll& 
        V^{1/2} \bigg[ N+\sum_{V<m\leq V'} \sum_{\substack{N/4V<n_{1}\leq N/V\\ N/4V<n_{2}\leq N/V\\ n_{1}\neq n_{2}}} e(\alpha(n_{1}^2-n_{2}^2)m^2+\beta (n_{1}-n_{2}) \bigg]^{1/2}\\
        \ll& 
        V^{1/2}\bigg[ N+\sum_{V<m\leq V'} \sum_{0<|k|\leq N/V} e(\beta k)\sum_{ N/4V<n\leq N/V}e(\alpha m^2k(k+2n)\bigg]^{1/2}.
    \end{align*}
    This is obtained by putting $k=n_{1}-n_{2}$ and $n=n_{2}$. Here the inner sum is over all $n \in I(m,k)$ where $I(m,k)$ is an interval of length $3N/4V$. This gives,
    \begin{align*}
        \sum \ll& V^{1/2} \bigg[ N+ \sum_{\substack{0<|k|\leq N/V\\ V<m\leq V'}}\min (3N/4V , ||\alpha m^2 k||^{-1})\bigg]^{1/2}.
    \end{align*}
    Putting $m^2k=\ell \leq NV$ and using Lemma \ref{Vino}, we have
    \begin{align*}
        \sum \ll N^{1/2}V+NV^{1/2}q^{-1/2}+(Vq)^{1/2}.
    \end{align*}
    This proves the Lemma.
\end{proof}


We shall also require a version averaged over multiples of $\alpha$.
Here is a consequence of part of Lemma~2 of \cite{Ghosh*81} by A.~Ghosh,
Eq.~(3.10).

\begin{lem}
  \label{HLinearGhosh}
  Assume that $|q\alpha-a|\le 1/q$ for some $a$ prime to~$q$. Then,
  for any complex sequence $(a_m)$ bounded in absolute value by~1, we have
  \begin{equation*}
    \sum_{h\le H}
    \biggl|\sum_{V< m\le V'}a_m
    \sum_{\substack{W<n\le W'\\ N/2 < mn\le N}}
    e(h\alpha (mn)^2)\biggr|
    \ll_\ve (Nq)^\ve
    \bigl( HNV^{1/2}q^{-1/2}+
    HN^{1/2}V+(HVq)^{1/2}
    \bigr).
  \end{equation*}
\end{lem}
Lemma~\ref{LinearGhosh} is in fact case $H=1$ of the above.
\begin{lem}
  \label{HLinear+congruence}
  Assume that $|q\alpha-a|\le 1/q$ for some $a$ prime to~$q$.Then,
  for any complex sequence $(a_m)$ bounded in absolute value by~1, we have
  \begin{equation*}
    \sum_{h\le H}
    \biggl|\sum_{V< m\le V'}a_m
    \sum_{\substack{W<n\le W'\\ N/2 < mn\le N\\ mn\equiv 1[4]}}
    e(h\alpha (mn)^2)\biggr|
    \ll_\ve (Nq)^\ve
    \bigl(HNV^{1/2}q^{-1/2}+
    HN^{1/2}V+(HVq)^{1/2}
    \bigr).
  \end{equation*}
\end{lem}
As Lemma~\ref{HLinearGhosh}, this is also a lemma which is analogous to the one of A. Ghosh but with a condition $mn \equiv 1[4]$. We treat the condition $mn\equiv 1[4]$ as in Lemma~\ref{linear+congruence} and then appeal to Lemma~\ref{HLinearGhosh}.

\section{Lemmas for bilinear sums}

 Let us have the parameters $V,V',W,W'$ such that $V,W \leq N$, and $V' \leq 2V$ and $W'\leq 2W$.

Here is a lemma which is analogous to Lemma~2 of \cite{Ghosh*81} by A.~Ghosh,
Eq.~(3.11). 
\begin{lem}
  \label{HBilinearGhosh}
  Assume that $|q\alpha-a|\le 1/q$ for some $a$ prime to~$q$. Then,
  for any complex sequences $(a_m)$ and $(b_n)$  bounded in absolute value by~1, we have
  \begin{align*}
    \sum_{H<h\le H'}
    \biggl|\sum_{V< m\le V'}a_m
    &\sum_{\substack{W<n\le W'\\ N/2 < mn\le N\\ mn \equiv 1 [4]}}b_n
    e(h\alpha (mn)^2)\biggr|
    \\\ll_\ve (Nq)^\ve 
    \min\biggl(
    &\frac{HN}{q^{1/4}}+\frac{HN}{W^{1/2}}+HN^{3/4}W^{1/4} +H^{3/4}N^{1/2}q^{1/4},
    \\&\frac{HN}{q^{1/4}}+H(NW)^{1/2}+\frac{HN}{W^{1/4}} +H^{3/4}N^{1/2}q^{1/4},
    \\&
    \biggl(
    \frac{HN}{q^{1/4}}
     +
     H^{3/4} (NW)^{1/2}
     +
     \frac{H N}{W^{1/4}}
     +
    H^{3/4} N^{1/2}q^{1/4}\biggr)
    \biggl(1+ \frac{H^{1/2}W^{1/3}}{N^{1/3}}\biggr)
    \biggr).
  \end{align*}
\end{lem}
The first two upper bounds are similar to the ones of Ghosh, on
exchanging $W$ and $N/W$. The third one is the one we prove
now. Before proving this lemma, let us state a simplification of it. 
\begin{lem}
    \label{HBilinear}
    When $W\ge V$, the second upper bound given in
    Lemma~\ref{HBilinearGhosh} is smaller than the first one.
    When $W\ge V$ and $H\le (N/W)^{2/3}$, the third one is the smallest. 
\end{lem}
It seems difficult to go farther. When $q$ is very
large, and $H$ is large,
the first upper bound is the smallest.
\begin{proof}[Proof of Lemma~\ref{HBilinear}]
When $W\ge V$, we have 
$HN/W^{1/2}\ll HN^{3/4}W^{1/4}$ simply because $W\gg N^{1/2}$ which
simplifies the first bound.  And since $HN^{3/4}W^{1/4}\gg
H(NW)^{1/2}$, we conclude that the second bound is the smallest of the
first two. The last statement is obvious.
\end{proof}
\begin{proof}[Proof of Lemma~\ref{HBilinearGhosh}]
   As before, we use the orthogonality property of Dirichlet characters to overcome the congruent condition, leading to
\begin{align*}
    \sum_{\substack{V< m\le V'\\W<n\le W'\\ N/2 < mn\le N\\ mn \equiv 1[4]}} a_mb_n e(h\alpha (mn)^2) 
    &=\frac12
    \sum_{\substack{V< m\le V'\\W<n\le W'\\ N/2 < mn\le N}}\sum_{\chi \mod{4}} \chi(m)\chi(n) a_mb_n e(h\alpha m^2n^2)
    \\&
    =\frac12\sum_{\chi \mod{4}}
    \sum_{\substack{V< m\le V'\\W<n\le W'\\ N/2 < mn\le N}} 
    (\chi(m)a_m)(\chi(n) b_n) e(h\alpha m^2n^2)
    . 
\end{align*}
This enables us to remove the congruence condition at no cost.
The next preparation step is to remove the condition $N/2<mn\le N$ via
Lemma~\ref{FourierVaughan} which we use with the parameters
$x=\log([N]+\frac12)$ and $\beta=\log mn$, so that  
\begin{equation*}
    \min(|\beta-x|,|\beta-x/2|)\gg 1/N.
\end{equation*}
We use $T=N^2$ getting an error term or $\Ocal(1)$.
Consequently
\begin{multline*}
    \sum_{\substack{V< m\le V'\\W<n\le W'\\ N/2 < mn\le N}} 
    (\chi(m)a_m)(\chi(n) b_n) e(h\alpha m^2n^2)
    \\
    \ll (\log N) \max_{|t|\le N^2}
    \biggl|\sum_{\substack{V< m\le V'\\W<n\le W'}} 
    (\chi(m)a_m m^{it})(\chi(n) b_n n^{it}) e(h\alpha m^2n^2)\biggr|
    +\Ocal(1).
\end{multline*}
Let us proceed by fixing $\chi$ and $t$, setting
$\tilde{a}_m=\chi(m)a_m m^{it}$ and similarly $\tilde{b}_n=\chi(n)b_n
n^{it}$. 
We are to majorize
\begin{align*}
    \Sigma
    &=\sum_{H<h\le H'}\biggl|\sum_{\substack{W<n\le W'\\ V<m\le
        V'}}\tilde{b}_n
    \tilde{a}_m e(\alpha h(mn)^2)\biggr|
    =\sum_{H<h\le H'}c_h\sum_{\substack{W<n\le W'\\ V<m\le
        V'}}\tilde{b}_n
    \tilde{a}_m e(\alpha h(mn)^2)
    \\&=\sum_{n\le 2W}\tilde{b}_n
    \sum_{\substack{L < \ell \leq 8L }} A_\ell e(\alpha \ell n^2)
\end{align*}
say, on setting $L=HN^2/(4W^2)$ and with
\begin{equation}
    A_\ell=\sum_{\substack{hm^2=\ell\\ H<h\le H'}}c_h\tilde{a}_m.
\end{equation}
 Now, by Cauchy's inequality,
\begin{align*}
     \Sigma^2 \ll&  
     \sum_{n\leq2W}|\tilde{b}_n|^2
     \sum_{n\leq2W}\biggl|\sum_{L< \ell\leq L'}
     A_\ell e(\alpha \ell n^2)\biggr|^2 \\ 
    \ll&  
     W
    \sum_{n\leq2W}\sum_{\ell\le L'}
    |A_\ell|^2
    +W\sum_{\substack{0< |j|\leq L'}}
    \sum_{\substack{L<\ell_1,\ell_2\le L'\\ j=\ell_1-\ell_2}}A_{\ell_{1}}A_{\ell_{2}} \sum_{n\leq2W}e(\alpha n^2j).
    \end{align*}
    The number of representations, say $R(j;H)$, of $j$ in the form
    $\ell_1-\ell_2$ weighted by~$A_{\ell_1}$ and $A_{\ell_2}$ is the
    number of solutions to $j=h_1m_1^2-h_2m_2^2$, with obvious size
    conditions on the variables.
    Lemma~\ref{Vaughan-Wooley} tells us that $R(j;H)\ll
    H^2N^\epsilon$, and therefore
    \begin{equation}
      \label{eq:1}
      \Sigma^2 \ll
      W
    \sum_{n\leq2W}\sum_{\ell\le L'}
    |A_\ell|^2
    +WH^2N^\epsilon\sum_{\substack{0< |j|\leq L'}}
    \biggl| \sum_{n\leq2W}e(\alpha n^2j)\biggr|.
    \end{equation}
    We record this inequality but we shall proceed differently. It
    would be enough when~$H=1$.
    Indeed the present proof proceeds by using the Cauchy inequality on $j$.
    The quantity $\sum_{0<|j|\le L'}|R(j;H)|^2$
     counts the number of solutions to
     $h_1m_1^2-h_2^2m_2^2=h_3m_3^2-h_4m_4^2$.
     Lemma~\ref{Browning-HB2}
     gives us the bound
    \begin{equation}
        \label{Bound4}
        \sum_{0<|j|\le L'}|R(j;H)|^2
        \ll_\epsilon H^3(N/W)^{2+\epsilon}+ H^5 (N/W)^{2/3}.
    \end{equation}
    We moreover check that 
    \begin{equation*}
    \sum_{\ell\leq L'}
    |A_\ell|^2
    \le
    \sum_{\substack{H<h_1\le H'\\ V<m_1\le V'}}
    d(h_1m_1^2)
    \ll_\epsilon H V N^\epsilon
    \ll_\epsilon HN^{1+\epsilon}W^{-1}.
    \end{equation*}
    Let us introduce these elements in the main proof. Applying Cauchy
    on the inner sum with $j$, we reach 
    \begin{align*}
        \Sigma^2
        \ll&
        N^\epsilon W\biggl[ HN
        + \bigl(H^{3/2} NW^{-1}+H^{5/2}(NW^{-1})^{1/3}\bigr)
        \biggl( \sum_{j\leq L'}
        \biggl| \sum_{n\leq 2W}e(\alpha n^2 j)\biggr|^2 \biggr)^{1/2} 
         \biggr]
        \\  \ll &  
        N^\epsilon HNW
                  +
                  N^\epsilon\bigl(H^{3/2} N+H^{5/2}N^{1/3}W^{2/3}\bigr)
        \biggl(
        \sum_{j\leq L'}
        W
        +\sum_{j\leq L'}\sum_{1\le |g|\le 4W^2}
        \sum_{\substack{n_1,n_2\le W'\\g= n_{1}^2 - n_{2}^2}} e(\alpha gj) \bigg)^{1/2}
        \\\ll &
        N^\epsilon HNW
        +N^\epsilon H^2 N^2 W^{-1/2}+N^\epsilon H^{3}N^{4/3}W^{1/6}
        \\&\qquad+N^\epsilon \bigl(H^{3/2} N+H^{5/2}N^{1/3}W^{2/3}\bigr)
        \biggl( \sum_{g\leq W^2}\min (L,||\alpha g||^{-1})\biggr)^{1/2}.
    \end{align*}
    Vinogradov's Lemma \ref{Vino} gives us the bound
    \begin{align*}
      \biggl( \sum_{g\leq W^2}\min (L,||\alpha g||^{-1})\biggr)^{1/2}
      &\ll
      \bigl(W^2Lq^{-1}+(W^2+q)\log 2q\bigr)^{1/2}
      \\&\ll
      \frac{H^{1/2}N}{q^{1/2}}
      +(W+\sqrt{q})\log(2q).
    \end{align*}
    We thus get
   \begin{align*}
     \frac{\Sigma^2}{(Nq)^\epsilon}
     \ll_\epsilon &
       HNW
        +H^2 N^{2}W^{-1/2}+ H^{3}N^{4/3}W^{1/6}
        + \frac{H^2N^2}{q^{1/2}}+\frac{H^{3}N^{4/3}W^{2/3}}{q^{1/2}}
        \\&   + H^{3/2} NW+H^{5/2}N^{1/3}W^{5/3}
           + H^{3/2} N\sqrt{q}+H^{5/2}N^{1/3}W^{2/3}\sqrt{q}.
   \end{align*}
   We see directly that the summand $HNW$ is superfluous. We may rewrite the bound
   in the form
   \begin{equation*}
     \frac{\Sigma^2}{(Nq)^\epsilon}
     \ll_\epsilon 
     \biggl(\frac{H^2N^2}{q^{1/2}}
     +
     H^2 N^{2}W^{-1/2}
     +
     H^{3/2} NW
     +
     H^{3/2} N\sqrt{q}\biggr)(1+ H(W/N)^{2/3}).
   \end{equation*}
   Therefore
   \begin{equation*}
     \frac{\Sigma}{(Nq)^\epsilon}
     \ll_\epsilon 
     \biggl(\frac{HN}{q^{1/4}}
     +
     \frac{H N}{W^{1/4}}
     +
     H^{3/4} \sqrt{NW}
     +
     H^{3/4} N^{1/2}q^{1/4}\biggr)(1+ H^{1/2}(W/N)^{1/3}).
   \end{equation*}
   This proves the lemma.
\end{proof}
The case $H=1$ in Lemma~\ref{HBilinearGhosh} gives us the next lemma.
\begin{lem}
  \label{Bilinear1}
  Let $a$ be prime to~$q$. Then, when $W\ge V/2$ and 
  for any complex sequences $(a_m)$ and $(b_n)$ bounded in absolute value by~1, we have
  \begin{equation*}
    \sum_{\substack{V< m\le V'\\W<n\le W'\\ N/2 < mn\le N\\ mn\equiv 1[4]}} a_mb_n
    e(\alpha(mn)^2)\ll_\ve (Nq)^\ve
    \biggl(
    \frac{N}{q^{1/4}}
    +(NW)^{1/2}+\frac{N}{W^{1/4}}
    +N^{1/2}q^{1/4}
    \biggr).
  \end{equation*}
\end{lem}


\section{Proof of Theorem~\ref{OneSum}}

\begin{proof}[Proof of Theorem~\ref{OneSum}]
   Let us start by considering the following decomposition of
   $S(\alpha, N)$ as linear and bilinear forms by Lemma~\ref{OR-KV}. We have 
       \begin{align}\label{main equation}
            S(\alpha,N) = \sum_{\substack{d<M\\d\mid P_{4,3}(z)}} \mu(d)\sum_{\substack{n\equiv 0(d)\\N/2 < n\leq N\\ n \equiv 1[4]}} e(\alpha n^2) - &\sum_{\substack{z \leq p\leq \sqrt{N}\\ N/2< mp \leq N \\mp \equiv 1[4]}}\rho(m) 1_{p\in P_{4,3}} e(\alpha m^2p^2)
            \notag\\
            - &\int_0^1 \sum_{\substack{k\ell\geq M\\M\leq \ell \leq Mz\\ N/2 < k\ell \leq N \\k\ell \equiv 1[4]}}\alpha_\ell(t)\beta_k(t) e(\alpha \ell^2k^2)dt +O(N/z). 
        \end{align}
        Here the first sum is linear, say $S_{1}$.  Let us localize $d$ in $D$ and $D'$, where $D' =\min\{ 2D, M\}$. Therefore we have,
        \begin{equation*}
            S_{1} \leq 2\bigg(\frac{\log M}{\log 2}+1\bigg)\max _{\substack{D\leq M\\ N/4D <T \leq N/D}}\bigg| \sum_{\substack{D<d\leq D'\\ d \mid P_{4,3}(z)}}\mu(d) \sum_{\substack{T< m \leq T'\\ md \equiv 1[4]\\ N/2< md \leq N}} e(\alpha (md)^{2})\bigg|
          \end{equation*}
          where m is localized in $T$ and $T' = \min \{ 2T, N/D\}$.
        We bound this sum using Lemma \ref{linear+congruence}, i.e,
        \begin{equation*}
            S_{1} \ll (Nq)^{\epsilon}\bigl( N^{1/2}M+NM^{1/2}q^{-1/2}+(Mq)^{1/2}\bigr).
        \end{equation*}
        Second and third terms of eq. \eqref{main equation} are bilinear sums say $S_{2}$ and $S_{3}$ respectively. In $S_{2}$, localizing $p$ in $(P,P']$ and $m$ in $(T,T']$, and in $S_{3}$, localizing $\ell$ in $(L,L']$ and $k$ in $(K,K']$ we get,
        \begin{equation*}
            S_{2} \leq 2\bigg(\frac{\log N}{\log 2} +1\bigg)\max_{\substack{z \leq P\leq \sqrt{N}\\ \sqrt{N}/2 \leq T\leq N/z}}\bigg| \sum_{\substack{P< p\leq P'\\T< m\leq T'\\ mp\equiv1[4]\\ N/2 < mp \leq N}}\rho(m)1_{p\in P_{4,3}}e(\alpha m^{2}p^{2})\bigg|
          \end{equation*}
          and
        \begin{equation*}
            S_3 \leq 2\bigg(\frac{\log Mz}{\log 2}+1\bigg) \max_{\substack{M\leq L\leq Mz\\ N/(2Mz) < K\leq N/M}}\bigg| \sum_{\substack{K< k\leq K'\\L< \ell \leq L'\\ N/2 < k\ell \leq N \\k\ell \equiv 1[4]}}\alpha_\ell(t)\beta_k(t) e(\alpha \ell^2k^2) \bigg|.
        \end{equation*}
        $S_2$ and $S_3$ may both be treated by Lemma~\ref{Bilinear1}:
        \begin{equation*}
          S_2+S_{3} \ll (Nq)^{\epsilon}
          \bigl(Nq^{-1/4}+N^{1/2}q^{1/4}
          +N^{7/8}+N^{1/2}\max(N/z, N/M)^{1/2}
          \bigr).
        \end{equation*}
   Combining the bounds for $S_{1},S_{2}$ and $S_3$, and by considering $q\leq N^{2}$ and $M\leq \sqrt{N}$ we have
\begin{align*}
  \frac{S(\alpha, N)}{N^{\epsilon}}
  \ll
  &\frac{N}{q^{1/4}}
  +N^{1/2}q^{1/4}
  + N^{7/8}
  \\&+N^{1/2}M+NM^{1/2}q^{-1/2}+(Mq)^{1/2}
  +\frac{N}{\min(z, M)^{1/2}}
  +Nz^{-1}.
\end{align*}
The last summand may be omitted, as well as $N^{1/2}M$ (compare to
$N^{7/8}$), and the best is to take $z=M$ so we have
\begin{equation*}
  \frac{S(\alpha, N)}{N^{\epsilon}}
  \ll
  \frac{N}{q^{1/4}}
  +N^{1/2}q^{1/4}
  + N^{7/8}
  +NM^{1/2}q^{-1/2}+(Mq)^{1/2}
  +\frac{N}{M^{1/2}}
.
\end{equation*}
\smallskip
\noindent{\bf When $q\le N$:} we have $N/q^{1/4}\ge N^{1/2}q^{1/4}$
and $(Mq)^{1/2}\le N/\sqrt{M}$. This calls for the choice $M=q^{1/2}$
which is indeed at most~$\sqrt{N}$,
getting the bound
\begin{equation*}
  \frac{S(\alpha, N)}{N^{\epsilon}}
  \ll
  \frac{N}{q^{1/4}}
  +N^{1/2}q^{1/4}
  + N^{7/8}\qquad(q\le N).
\end{equation*}
\smallskip
\noindent{\bf When $N\le q\le N^2$:} we have $N/q^{1/4}\le N^{1/2}q^{1/4}$
and $N(M/q)^{1/2}\le N/\sqrt{M}$. This calls for the choice $M=N/q^{1/2}$
which is indeed at most~$\sqrt{N}$,
getting again the bound
\begin{equation*}
  \frac{S(\alpha, N)}{N^{\epsilon}}
  \ll
  \frac{N}{q^{1/4}}
  +N^{1/2}q^{1/4}
  + N^{7/8}\qquad(N\le q\le N^2).
\end{equation*}
This ends the proof of our theorem.
\end{proof}

\section{Proof of Theorem~\ref{SeveralSum}}
\begin{proof}[Proof of Theorem~\ref{SeveralSum}]
    By Lemma \ref{OR-KV}, we have
       \begin{align}\label{hmain}
           \sum_{h\leq H} | S(h\alpha,N)| 
           =& \sum_{h\leq H} 
           \sum_{\substack{d<M\\d\mid P_{4,3}(z)}} \mu(d)
           \sum_{\substack{n\equiv 0(d)\\ n\equiv 1[4]\\
           N/2 < n \leq N}} e(\alpha h n^2) 
           \notag\\&- \sum_{h\leq H}\sum_{\substack{N/2 <mp \leq N\\z \leq p\leq \sqrt{N}\\ mp \equiv 1[4]}}\rho(m) 1_{p\in P_{4,3}} 
           e(\alpha h m^2p^2)
           \notag\\&
           - \int_0^1 \sum_{h\leq H}
           \sum_{\substack{k\ell\geq M\\M\leq \ell \leq Mz \\ N/2 < k\ell \leq N \\ k\ell \equiv 1[4]}}
           \alpha_\ell(t)\beta_k(t) e(\alpha h \ell^2k^2) +O({HN}/{z}). 
        \end{align}
     The first sum here is linear, say $S_{1}$.  Let us localize $d$ in $D$ and $D'$, where $D' =\min\{ 2D, M\}$ and m is localized in $T$ and $T' = \min \{ 2T, N/D\}$. Therefore we have,
        \begin{equation*}
            S_{1} \leq 2 \bigg( \frac{\log M}{\log 2}+1 \bigg)  \max _{\substack{D\leq M\\N/4D <  T \leq N/D}}\sum_{h\leq H}\bigg| \sum_{\substack{D<d\leq D'\\ d \mid P_{4,3}(z)}}\mu(d) \sum_{\substack{T< m \leq T'\\ md \equiv 1[4]\\ N/2 < md \leq N}} e(\alpha h(md)^{2})\bigg|.
        \end{equation*}
We bound this sum using Lemma \ref{HLinear+congruence}, i.e,
        \begin{equation*}
            S_{1} \ll (Nq)^{\epsilon}\bigg( HN^{1/2}M+HNM^{1/2}q^{-1/2}+(HMq)^{1/2}\bigg).
        \end{equation*}
        Second and third terms of Eq. \eqref{hmain} are bilinear sums say $S_{2}$ and $S_{3}$ respectively. In $S_{2}$, localizing $p$ in $(P,P']$ and $m$ in $(T,T']$, and in $S_3$, localizing $\ell$ in $(L,L']$ and $k$ in $(K,K']$ we get,
        \begin{equation*}
            S_{2} 
            \leq 
            2\bigg(\frac{\log N}{\log 2} +1\bigg)
            \max_{\substack{z \leq P\leq \sqrt{N}\\ \sqrt{N}/2 \leq T\leq N/z}}
            \sum_{h\leq H}\bigg| \sum_{\substack{P< p\leq P'\\T< m\leq T'\\ mp\equiv1[4]\\ N/2 < mp \leq N}}
            \rho(m)1_{p\in P_{4,3}}e(\alpha hm^{2}p^{2})\bigg|
        \end{equation*}
        and
        \begin{equation*}
            S_3 \leq 2\bigg(\frac{\log Mz}{\log 2}+1\bigg) \max_{\substack{M\leq L\leq Mz\\ N/(2Mz) < K\leq N/M}}\sum_{h\leq H}\bigg| \sum_{\substack{K< k\leq K'\\L< \ell \leq L'\\ N/2 < k\ell \leq N \\k\ell \equiv 1[4]}}\alpha_\ell(t)\beta_k(t) e(\alpha h\ell^2k^2) \bigg|.
          \end{equation*}
        In $S_2$, we have $P/2 \leq T$. Considering $H \leq z^{2/3}\le
        (NT^{-1})^{2/3}$ in the sum say $S_2$ and by applying the third
        bound from Lemma \ref{HBilinearGhosh} we have, 
        \begin{align*}
          S_{2} \ll (Nq)^{\epsilon}
          \big( HNq^{-1/4}+H^{3/4}Nz^{-1/2}+HN^{7/8}+H^{3/4}N^{1/2}q^{1/4}
            \big).
        \end{align*}
        In case we do not have $H \leq z^{2/3}$, then we cannot ensure
        that $H\le (N/T)^{2/3}$, so we use the second bound from Lemma~\ref{HBilinearGhosh}:
        \begin{align*}
          S_2 \ll (Nq)^\epsilon
          \big( HNq^{-1/4}+HNz^{-1/2}+HN^{7/8}+H^{3/4}N^{1/2}q^{1/4}\big).
        \end{align*}
        On considering $S_2$ and $S_3$, we readily see that there is
        no point in taking~$M\ge z$, so we take $M=z$ and get the
        same bounds for $S_3$ as for $S_2$.

        Combining the bounds for $S_1$, $S_2$ and $S_3$, and by
        considering $q\leq N^{2}$, we get 
\begin{align*}
  \sum_{h\leq H}\frac{|S(\alpha h ,N)|}{N^{\epsilon}}
  \ll&
       HN^{1/2}z+HNz^{1/2}q^{-1/2}+(Hzq)^{1/2}+ HNq^{-1/4}
       \\&+HNz^{-1/2} +HN^{7/8}+H^{3/4}N^{1/2}q^{1/4}  
\end{align*}
where the summand $HNz^{-1/2}$ may be reduced to $H^{3/4}Nz^{-1/2}$
when $H \le  z^{2/3} $. We have $HN^{1/2}z\le HN^{7/8}$, which enables
us to discard the first summand.
We may thus rewrite the above bound in the form
\begin{align}
  \label{ref}
  \sum_{h\leq H}\frac{|S(\alpha h ,N)|}{N^{\epsilon}}
  \ll&
       HNq^{-1/4}+HNz^{1/2}q^{-1/2}+(Hzq)^{1/2}
       \\\notag&+(H \text{ or } H^{3/4})Nz^{-1/2} +HN^{7/8}+H^{3/4}N^{1/2}q^{1/4},
\end{align}
the trivial bound being $HN$.
\smallskip

\noindent\textbf{When $q\le N$:} Combining $HNz^{1/2}q^{-1/2}$ and $HNz^{-1/2}$
we get the choice $z=q^{1/2}$ which is at most~$\sqrt{N}$.
    \begin{equation*}
      \sum_{h\leq H}\frac{|S(h\alpha,N)|}{N^{\epsilon}}
      \ll HNq^{-1/4}+HN^{7/8}
      +H^{1/2}q^{3/4}+H^{3/4}N^{1/2}q^{1/4}.
    \end{equation*}
    The third summand is less than the fourth one and the fourth one
    is not more than the second one, so we get
    \begin{equation*}
      \sum_{h\leq H}\frac{|S(h\alpha,N)|}{N^{\epsilon}}
      \ll HNq^{-1/4}+HN^{7/8},
      \quad(q\le N).
    \end{equation*}
    This bound is good for small values of $q$. It is optimal with
    respect to what we may expect to infer from~\eqref{ref}. 
    When $H$ is smaller than $z^{2/3}$, we may balance
    $HNz^{1/2}q^{-1/2}$ together with $H^{3/4}Nz^{-1/2}$, leading to
    the choice $z=q^{1/2}/H^{1/4}$, which would be legal when $H\le
    q^{2/7}$. However this smaller choice of $z$ would not help as
    the summand $HNq^{-1/4}$ will dominate. This is the bound when
    $q\le N$.
    
\smallskip\noindent\textbf{When} comparing $HN^{1/2}z$ and
$HNz^{-1/2}$ we get the choice  $z=N^{1/3}$.
    \begin{equation*}
      \sum_{h\leq H}\frac{|S(h\alpha,N)|}{N^\epsilon}
      \ll
      HNq^{-1/4}
      +HN^{7/6}q^{-1/2}
      +H^{1/2}N^{1/6}q^{1/2}
      +HN^{7/8}+H^{3/4}N^{1/2}q^{1/4}.
    \end{equation*}
    We may further assume $q>N$ since we already an ``optimal'' bound otherwise.
    As a consequence the first summand is dominated by $HN^{7/8}$ and
    so is the second one, getting
    \begin{equation*}
      \sum_{h\leq H}\frac{|S(h\alpha,N)|}{N^\epsilon}
      \ll
      H^{1/2}N^{1/6}q^{1/2}
      +HN^{7/8}+H^{3/4}N^{1/2}q^{1/4},\quad (q\ge N).
    \end{equation*}
    This gives the bound
    \begin{equation*}
      \sum_{h\leq H}\frac{|S(h\alpha,N)|}{N^\epsilon}
      \ll
      \begin{cases}
        HN^{7/8},\text{when $N\le q\le HN^{17/12}$},\\
        H^{1/2}N^{1/6}q^{1/2},\text{when $ q\ge HN^{17/12}$}.
      \end{cases}
    \end{equation*}
    
    When $H$ is small, we have to compare $HN^{1/2}z$ and
    $H^{3/4}Nz^{-1/2}$, getting the choice  $z=N^{1/3}/H^{1/6}$. This
    is legal when $H\le N^{2/5}$, getting
    \begin{equation*}
      \sum_{h\leq H}\frac{|S(h\alpha,N)|}{N^\epsilon}
      \ll
      HNq^{-1/4}
      +H^{11/12}N^{7/6}q^{-1/2}
      +H^{1/2}N^{1/6}q^{1/2}
      +HN^{7/8}+H^{3/4}N^{1/2}q^{1/4}.
    \end{equation*}
    We may again discard $HNq^{-1/4}$ and also
    $H^{11/12}N^{7/6}q^{-1/2}$ by comparing them with~$HN^{7/8}$ so we
    are back with the first case.

\smallskip\noindent\textbf{We finally } may compare $(Hzq)^{1/2}$ and
$HNz^{-1/2}$, leading to the choice $z=NH^{1/2}q^{-1/2}$.
This choice is only valid for $q \geq HN$ to ensure the condition
$z\leq \sqrt{N}$, but since we already have the ``optimal'' bound when
$q\le HN^{17/12}$, we may assume $q>HN^{17/12}$.
Thus for $HN^{17/12} \leq q \leq N^2$, we have the bound
\begin{align*}
  \sum_{h\leq H}\frac{|S(h\alpha,N)|}{N^\epsilon}
  &\ll
  H^{5/4}N^{3/2}q^{-3/4}+HN^{7/8}+H^{3/4}N^{1/2}q^{1/4}
  \\&\ll
  H^{5/4}N^{3/2}q^{-3/4}+H^{3/4}N^{1/2}q^{1/4}
  \\&\ll
  H^{3/4}N^{1/2}q^{1/4}.
\end{align*}
Yet again we may compare instead $(Hzq)^{1/2}$ and
$H^{3/4}Nz^{-1/2}$, leading to the choice $z=NH^{1/4}q^{-1/2}$.
This is legal when $q\ge NH^{1/2}$ and $H^{3/2}\le NH^{1/4}q^{-1/2}$,
i.e. $H\le N^{4/5}/q^{2/5}$. We obtain the bound
\begin{align*}
  \sum_{h\leq H}\frac{|S(\alpha h ,N)|}{N^{\epsilon}}
  &\ll
    H^{9/8}N^{3/2}q^{-3/4}+H^{5/8}N^{1/2}q^{1/4}+HN^{7/8}+H^{3/4}N^{1/2}q^{1/4}
  \\&\ll
  H^{9/8}N^{3/2}q^{-3/4}+HN^{7/8}+H^{3/4}N^{1/2}q^{1/4}
  \ll H^{3/4}N^{1/2}q^{1/4}.
\end{align*}

  \bigskip

Now by combining these, we have
$$\sum_{h\leq H}\frac{|S(h\alpha,N)|}{N^\epsilon} \ll \begin{cases}
     HNq^{-1/4} \quad  \quad \quad \text{if } q \leq N^{1/2}\\
     HN^{7/8} \quad \quad \quad \quad \text{if } \ \ N^{1/2} < q\leq HN^{17/12}\\
     H^{1/2}N^{1/6}q^{1/2} \quad \quad \quad \quad \text{if } \ \ HN^{17/12} < q\leq N^2\\
     H^{3/4}N^{1/2}q^{1/4} \quad \text{if }\ \ HN^{17/12} < q \leq N^2.
\end{cases}$$\\
This can be slightly degraded and combined together to give a single estimate as,
\begin{equation*}
    \sum_{h\leq H} \frac{|S(h\alpha,N)|}{N^\epsilon} \ll  HNq^{-1/4}+ HN^{7/8}+H^{3/4}N^{1/2}q^{1/4}.
\end{equation*}
This proves the theorem as, in the case $q\ge N^2$, the proposed bound is trivial.
\end{proof}

The research work of the first author is supported by Raman-Charpak Fellowship. We thank Aix Marseille Universit\'e and CNRS laboratory for the hospitality and the research facilities they provided.

\printbibliography

\end{document}